\documentclass[oneside,british,english,a4paper,11pt,makeidx]{amsart}
\usepackage[T1]{fontenc}
\usepackage[latin9]{inputenc}
\usepackage{verbatim}
\usepackage{float}
\usepackage{units}
\usepackage{textcomp}
\usepackage{multirow}
\usepackage{amsthm}
\usepackage{amstext}
\usepackage{amssymb}
\usepackage[all]{xy}
\usepackage[unicode=true,
 bookmarks=true,bookmarksnumbered=false,bookmarksopen=false,
 breaklinks=false,pdfborder={0 0 0},backref=false,colorlinks=true]
 {hyperref}
\hypersetup{pdftitle={Pseudoeffective cones in 2-Fano varieties},
 pdfauthor={ Muratore Giosu\`{e} Emanuele},
 pdfsubject={Algebraic Geometry}}
\makeatletter

\providecommand{\tabularnewline}{\\}

\numberwithin{equation}{section}
\numberwithin{figure}{section}
\theoremstyle{plain}
\newtheorem{thm}{\protect\theoremname}[section]
  \theoremstyle{definition}
  \newtheorem{defn}[thm]{\protect\definitionname}
  \theoremstyle{plain}
  \newtheorem{conjecture}[thm]{\protect\conjecturename}
  \theoremstyle{remark}
  \newtheorem{rem}[thm]{\protect\remarkname}
  \theoremstyle{plain}
  \newtheorem{lem}[thm]{\protect\lemmaname}
  \theoremstyle{plain}
  \newtheorem{prop}[thm]{\protect\propositionname}

\newtheorem*{Notation*}{Notation}

\usepackage{geometry}

\usepackage{enumitem}
\newenvironment{myitemize}
    {\begin{itemize}[leftmargin=*]
\addtolength{\leftmargin}{0in}
}
    {\end{itemize}}

\makeatother

\usepackage{babel}
  \providecommand{\conjecturename}{Conjecture}
  \providecommand{\definitionname}{Definition}
  \providecommand{\lemmaname}{Lemma}
  \providecommand{\propositionname}{Proposition}
  \providecommand{\remarkname}{Remark}
\providecommand{\theoremname}{Theorem}

\begin{document}
\global\long\def\C{\mathbb{\mathbb{C}}}
\global\long\def\R{\mathbb{\mathbb{R}}}
\global\long\def\Q{\mathbb{\mathbb{Q}}}
\global\long\def\Z{\mathbb{\mathbb{Z}}}
\global\long\def\N{\mathbb{\mathbb{N}}}
\global\long\def\k{\Bbbk}
\global\long\def\bP{\mathbb{P}}

\global\long\def\e{\epsilon}
 \global\long\def\O{\mathcal{O}}
\global\long\def\dual{\!\vee}
\global\long\def\I{\mathcal{I}}

\global\long\def\Exc#1{\mathrm{Exc}(#1)}

\global\long\def\Effc#1#2{\overline{\mathrm{Eff}}_{#1}(#2)}
\global\long\def\EffcX#1{\overline{\mathrm{Eff}}(#1)}
\global\long\def\Eff#1#2{\mathrm{Eff}_{#1}(#2)}
\global\long\def\Neff#1#2{\mathrm{Nef}_{#1}(#2)}
\global\long\def\Nef#1{\mathrm{Nef}(#1)}

\global\long\def\Hom{\mathit{\mathrm{Hom}}}
\global\long\def\Ker{\mathrm{Ker}}
\global\long\def\rk{\mathrm{rk}}

\global\long\def\Num{\mathrm{Num}}
\global\long\def\Rat{\mathrm{Rat}}
\global\long\def\Alg{\mathrm{Alg}}
\global\long\def\P#1{\mathbb{\mathbb{P}}^{#1}}
 \global\long\def\Pn{\mathbb{\mathbb{P}}^{n}}
\global\long\def\Pw{\mathbb{\mathbb{P}}(\mathbf{w})}

\global\long\def\g{\mathfrak{g}}
\global\long\def\h{\mathfrak{h}}
\global\long\def\b{\mathfrak{b}}
\global\long\def\p{\mathfrak{p}}

\global\long\def\Yb{Y_{b}}
\global\long\def\Xb{X_{b}}
\global\long\def\h#1#2{h^{#1,#2}}
\global\long\def\ra{\rightarrow}
\global\long\def\c{\chi}
\global\long\def\U{\mathcal{U}}
\global\long\def\s{\sigma}
\global\long\def\v{\!\vee}

\title{Betti numbers and pseudoeffective cones in 2-Fano varieties}

\author{Giosu\`{e} Emanuele Muratore}

\date{13 November 2017}
\begin{abstract}
The 2-Fano varieties, defined by De Jong and Starr, satisfy some higher
dimensional analogous properties of Fano varieties. We propose a definition
of (weak) $k$-Fano variety and conjecture the polyhedrality of the
cone of pseudoeffective {\normalsize{}$k$}-cycles for those varieties
in analogy with the case $k=1$. Then, we calculate some Betti numbers
of a large class of $k$-Fano varieties to prove some special case
of the conjecture. In particular, the conjecture is true for all 2-Fano
varieties of index $\ge n-2$, and also we complete the classification
of weak 2-Fano varieties answering Questions 39 and 41 in \cite{MR3114934}.
\end{abstract}

\address{Universit\`{a} Degli Studi Roma Tre\\
Dipartimento di Matematica e Fisica\\
Largo San Murialdo 1, 00146 Roma Italy.}

\address{Universit\'{e} de Strasbourg, CNRS\\
IRMA\\
7 Rue Ren\'{e} Descartes, 67000 Strasbourg France.}

\email{gmuratore@mat.uniroma3.it}


\maketitle

\section{Introduction}

The study of cones of curves or divisors on smooth complex projective
varieties $X$ is a classical subject in Algebraic Geometry and is
still an active research topic. However, little is known when we pass
to higher dimensions. For example it is a classical result that the
cone of nef divisors is contained in the cone of pseudoeffective divisors,
but in general $\Neff kX\subseteq\Effc kX$ is not true. These phenomena
can appear only if $\dim X\geq4$ and very few examples are known.
In particular \cite{MR2862063} gives two examples of such varieties.
Furthermore \cite{Ottem:inpress} proves that if $X$ is the variety
of lines of a very general cubic fourfold in $\P 5$, then the cone
of pseudoeffective 2-cycles on $X$ is strictly contained in the cone
of nef 2-cycles.

The central subject of this paper will be the $k$-Fano varieties.
\begin{defn}
A smooth Fano variety $X$ is $k$-Fano if the $s^{th}$ Chern character
$ch_{s}(X)$ is positive (see Definition \ref{def positive, nef})
for $1\le s\le k$, and weak $k$-Fano for $k>1$ if $X$ is $(k-1)$-Fano
and $ch_{k}(X)$ is nef.
\end{defn}
There is a large interest in studying varieties with positive Chern
characters. For example varieties with positive $ch_{1}(X)$ are Fano,
hence uniruled, that is there is a rational curve through a general
point. Fano varieties with positive second Chern character were introduced
by J. de Jong and J. Starr in \cite{dJS:inpress,MR2322679}. They
proved a (higher dimensional) analogue of this result: weak 2-Fano
varieties of pseudo-index at least 3 have a rational surface through a general point. Furthermore
if $X$ is weak 3-Fano then there is a rational threefold through
a general point of $X$ (under some hypothesis on the polarized minimal
family of rational curves through a general point of $X$, \cite[Theorem 1.5(3)]{MR2876140}).

Another problem concerns how the geometry of the cones of pseudoeffective
$k$-cycles depends on the positivity of the Chern characters $ch_{s}(X)$.
Mori\textquoteright{}s Cone Theorem resolves this problem for $k=1$:
the positivity of $ch_{1}(X)$ implies the polyhedrality of the cone
of pseudoeffective 1-cycles and the extremal rays are spanned by classes
of rational curves. By Kleiman\textquoteright{}s Theorem, a variety
with positive $ch_{1}(X)$ is just a Fano variety, that is with $c_{1}(X)$
ample, but this is not enough, in general, for the polyhedrality of
cones of pseudoeffective $k$-cycles for $k>1$: Tschinkel showed
a Fano variety where $\Eff 2X$ has infinitely many extremal rays.
Therefore more positivity is needed in order to obtain polyhedrality
of cones of pseudoeffective $k$-cycles for $k>1$.

In this paper we investigate a possible way of generalizing Mori\textquoteright{}s
result: 
\begin{conjecture}
\label{Problem}If $X$ is $k$-Fano, then $\Effc kX$ is a polyhedral
cone.
\end{conjecture}
The computing of the fourth Betti number is enough to show the polyhedrality
of some of the cones of 2-cycles for a large class of varieties: complete
intersections in weighted projective spaces, rational homogeneous
varieties and most complete intersections in them, etc. This allows
us to test the conjecture for many 2-Fano varieties, and in particular
we prove that it holds for del Pezzo and Mukai varieties. Using the
classification of Araujo-Castravet, we also prove the following.
\begin{thm}\label{Teorema1}
Let $X$ be a $n$-dimensional 2-Fano variety with $i_{X}\ge n-2$. Then
\textup{$\Effc 2X$} is  polyhedral. Also, \textup{$\Effc 3X$} is polyhedral with the possible exception of the complete intersection of type $(2,2)$ in $\P8$.

In particular, Conjecture \ref{Problem} is true for any $n$-dimensional
$k$-Fano variety with $i_{X}\ge n-2$ and $k=2,3$.
\end{thm}
Let $X$ be a complete intersection in $G(2,5)$ or $G(2,6)$ with
two hyperplanes under the Pl\"{u}cker embedding.  Araujo and Castravet
proved that $X$ is not 2-Fano, but questioned if it is weak 2-Fano
\cite[Proposition 32 and Questions 39,41]{MR3114934}. In \cite[Corollary 5.1]{DeArruda}
it is proved that a general such $X$ is not 2-Fano by showing that
there exists an effective surface $S$ such that $[i(S)]=\sigma_{1,1}^{\vee}$,
where $i$ is the inclusion. In this circumstance we can prove that
all the smooth complete intersections of this type are not weak 2-Fano,
and this completes the classification given in \cite[Theorem 3 and 4]{MR3114934}.
\begin{thm}
\label{thm:Gen De Arruda}Let $Y=G(2,5)$ or $G(2,6)$, let $X$ be
a smooth complete intersection of type $(1,1)$ in $Y$ under the
Pl\"{u}cker embedding. Then $X$ is not weak 2-Fano.
\end{thm}
These ideas can be improved in three very promising directions: to
generalize Tschinkel\textquoteright{}s example to higher dimensions,
to prove the conjecture for some Fano 4-folds of index 1, and to use
minimal families of rational curves to prove the conjecture for other
2-Fano\textquoteright{}s.

\thanks{I thank Angelo Lopez for all the support he has shown me since the
beginning of this work, and Gianluca Pacienza for his help. I also
thank Carolina Araujo, Izzet Coskun and Enrico Fatighenti for answering
many of my questions. }

\section{General facts about cycles}

A variety is a reduced and irreducible algebraic scheme over $\C$.

Throughout this paper we will use the following.

\begin{Notation*}$\hspace{1em}$

\begin{myitemize}

\item $X$ is a variety of dimension $n\ge4$.

\item$k$ is an integer such that $1\le k\le n-1$.

\item$H_{i}(X,G)$ and $H^{i}(X,G)$ are the singular homology and
cohomology groups of $X$ for $1\le i\le2n$ and coefficients in a
group $G$.

\item$b_{i}(X)$ is the $i^{th}$ Betti number of X for $1\le i\le2n$,
that is the rank of $H_{i}(X,\Z)$ or of $H^{i}(X,\Z)$.

\item$Z_{k}(X)$ is the group of $k$-cycles with integer coefficients.

\item$\Rat_{k}(X)$ is the group of $k$-cycles rationally equivalent
to zero.

\item$A_{k}(X)$ is the Chow group of $k$-cycles on $X$, that is
$A_{k}(X)=\nicefrac{Z_{k}(X)}{\Rat_{k}(X)}$.

\item$A_{*}(X)=\bigoplus_{k=0}^{n}A_{k}(X)$ is the Chow ring of
$X$.

\item$\mathrm{Alg}_{k}(X)$ is the group of $k$-cycles algebraically
equivalent to zero.

\item$\Hom_{k}(X)$ is the group of $k$-cycles homologically equivalent
to zero, that is the kernel of the cycle map $cl:Z_{k}(X)\rightarrow H_{2k}(X,\Z)$.

\item$\Num_{k}(X)$ is the group of cycles numerically equivalent
to zero, that is the group of cycles $\alpha\in Z_{k}(X)$ such that
$P\cdot cl(\alpha)=0$ for all polynomials $P$ in Chern classes of
vector bundles on $X$.

\item$N_{k}(X)$ is the quotient group $\nicefrac{Z_{k}(X)}{\Num_{k}(X)}$,
and $N_{k}(X)_{\R}:=N_{k}(X)\otimes\R$.

\item$\Eff kX\subseteq N_{k}(X)_{\R}$ is the cone generated by numerical
classes of effective $k$-cycles.

\item  Let $s\ge1$ be an integer. The $s^{th}$ Chern character of $X$, $ch_{s}(X)$, is the homogeneous part of degree $s$ of the total Chern character of $X$. For example, if $c_{i}(X)$ are the Chern classes of $X$, then $ch_{1}(X)=c_{1}(X)$, $ch_{2}(X)=\frac{1}{2}(c_{1}^{2}(X)-2c_{2}(X))$,  $ch_{3}(X)=\frac{1}{6}(c_{1}^{3}(X)-4c_{1}(X)c_{2}(X)+3c_{3}(X))$

\end{myitemize}\end{Notation*}

We will often use the following well-known facts:
\begin{rem}
There is a chain of inclusions \cite[p.374]{MR1644323}

\[
\Rat_{k}(X)\subseteq\mathrm{Alg}_{k}(X)\subseteq\Hom_{k}(X)\subseteq\Num_{k}(X)\subseteq Z_{k}(X)
\]
that gives rise to a diagram
\begin{equation}
\xymatrix{A_{k}(X)\ar@{->>}[r] & \nicefrac{Z_{k}(X)}{\mathrm{Alg}_{k}(X)}\ar@{->>}[r] & \nicefrac{Z_{k}(X)}{\Hom_{k}(X)}\ar@{^{(}->}[d]\ar@{->>}[r]^{\quad\;\pi_{k}} & N_{k}(X)\\
 &  & H_{2k}(X,\Z)
}
\label{eq:Diagramma}
\end{equation}
We set 
\begin{equation}
\pi_{k,\R}:\nicefrac{Z_{k}(X)}{\Hom_{k}(X)}\otimes\R\twoheadrightarrow N_{k}(X)_{\R}\label{eq:pi_k_R}
\end{equation}
the tensor product of $\pi_{k}$ and $id_{\R}$.
\end{rem}

\begin{rem}
By linearity of the intersection product, $N_{k}(X)$ is torsion free.
When $X$ is smooth, the intersection product gives a perfect pairing
\cite[Definition 19.1]{MR1644323} 
\[
N_{k}(X)_{\R}\otimes N_{n-k}(X)_{\R}\rightarrow\R.
\]

\end{rem}

\begin{defn}
\label{def positive, nef}Let $X$ be a smooth variety. A class $\alpha\in N_{k}(X)_{\R}$
is \emph{positive} if $\alpha\cdot\beta>0$ for every $\beta\in\Effc{n-k}X\backslash\{0\}$,
and it is \emph{nef} if $\alpha\cdot\beta\ge0$ for every $\beta\in\Effc{n-k}X$.
The cone generated by nef classes of $k$-cycles is $\Neff kX$.
\end{defn}
Kleiman\textquoteright{}s criterion for amplitude \cite[Theorem 1.4.29]{MR2095471}
states that the cone of positive $(n-1)$-cycles is exactly the cone
of numerical classes of ample divisors.
\begin{lem}
\label{lem:Lemma 1}Let $X$ be a projective variety. Then
\begin{enumerate}
\item If either $\rk A_{k}(X)=1$ or $b_{2k}(X)=1$, then $\Effc kX$ is
a half-line.
\item If either $\rk A_{k}(X)=2$ or $b_{2k}(X)=2$, then $\Effc kX$ is
either a half-line or it is spanned by two extremal rays.
\end{enumerate}
\end{lem}
\begin{proof}
In the first case, by diagram (\ref{eq:Diagramma}), we have a surjection
$\Z\twoheadrightarrow N_{k}(X)$ and, as $N_{k}(X)$ is torsion-free,
it must be $N_{k}(X)\cong\Z$. In the second case, again by diagram
(\ref{eq:Diagramma}), there is a surjection $\Z^{2}\twoheadrightarrow N_{k}(X)$
and then either $N_{k}(X)\cong\Z$ or $N_{k}(X)\cong\Z^{2}$. Since
$\Effc kX$ generates $N_{k}(X)_{\R}$, it is either a half-line or
it is spanned by two extremal rays, depending on the rank of $N_{k}(X)_{\R}$.\end{proof}
\begin{rem}
In a general, a variety $X$ with $ch_{k}(X)$ positive may not be
$k$-Fano. For example, in \cite{MR527834} Mumford found a smooth
surface $S$ of general type with $ch_{2}(S)=\frac{3}{2}$.
\end{rem}

\section{Cycles on Fano Varieties}

We study here the pseudoeffective cones of $k$-cycles on some well-known
classes of Fano varieties.

\subsection{Weighted projective spaces}

Let $\Pw$ be the weighted projective space where $\mathbf{w}=(w_{0},...,w_{n})\in\N_{0}^{n}$.
\begin{prop}
\label{prop:3}Let $X$ be a $n$-dimensional smooth complete intersection
in a weighted projective space. If $k\neq\frac{n}{2}$ then $b_{2k}(X)=1$.
In particular $\Effc kX$ is polyhedral.\end{prop}
\begin{proof}
Recall \cite[B13]{MR1194180} that $\dim H^{2i}(\Pw,\Q)=1$ for every
$0\le i\le\dim\Pw$. By Lefschetz's Hyperplane Theorem \cite[B22]{MR1194180}
we have that $H^{2k}(X,\Q)\cong H^{2k}(\Pw,\Q)$ for $2k<n$, then
$b_{2k}(X)=1$ for $k<\frac{n}{2}$. But $b_{2n-2k}(X)=b_{2k}(X)$,
then it follows that, for $k\neq\frac{n}{2}$, $b_{2k}(X)=1$ and
by Lemma \ref{lem:Lemma 1} that $\Effc kX$ is a half-line.
\end{proof}
Furthermore, if $X$ is a $k$-Fano complete intersection in a projective
space, then we can solve Conjecture \ref{Problem}, even for weak
Fano.
\begin{thm}
\label{thm:Conj for comp int in Pn}Let $X$ be a $n$-dimensional
weak $k$-Fano complete intersection in a projective space. If $1\le s\le k$,
then $b_{2s}(X)\le2$. In particular $\Effc sX$ is polyhedral.\end{thm}
\begin{proof}
Let $X$ be of type $(d_{1},...,d_{c})$ in $\bP^{n+c}$, with $d_{i}\ge2$
for $1\le i\le c$. By Proposition \ref{prop:3}, we can suppose $n$
even and $s=\frac{n}{2}$. We know from \cite[3.3.1]{MR3114934} that
$ch_{\frac{n}{2}}(X)$ is nef if and only if $d_{1}^{\frac{n}{2}}+...+d_{c}^{\frac{n}{2}}\le n+c+1$.
Since $n\ge4$, it follows easily that $c=1$. On the other hand $d_{1}^{\frac{n}{2}}\le n+2$
is possible only for $d_{1}=2$, that is $X$ is an $n$-dimensional
quadric. But $b_{n}(X)=2$ \cite[p.20]{Reid:thesis} and the theorem
follows by Lemma \ref{lem:Lemma 1}.
\end{proof}

\subsection{Rational homogeneous varieties}

Let $G$ be a reductive linear algebraic group defined over $\C$,
$B$ a Borel subgroup of $G$. We consider the set of simple $B$-positive
roots and denote by $S$ the corresponding set of reflections in the
Weyl group $W$. Then the pair $(W,S)$ is a Coxeter system in the
sense of \cite[Chapitre IV, D\'{e}finition 3]{MR0240238}. Let $l:W\rightarrow\N_{0}$
be the length function relative to the system $S$ of generators of
$W$. Furthermore we fix a subset $\Theta$ of $S$ and denote by
$W_{\Theta}$ the subgroup of $W$ generated by $\Theta$ and by $P$
a subgroup of $G$ associated to $\Theta$. Then the quotient $G/P$
is a projective variety, which is called a rational homogeneous variety.
Any rational homogeneous variety is a Fano variety \cite{MR0102800},
and the action of $G$ on $G/P$ by left multiplication is transitive.
Let $w_{0}$ (respectively, $w_{\theta}$) be the unique element of
maximal length of $W$ (respectively, $W_{\Theta}$). A simple calculation
shows that $\dim G/P=l(w_{0})-l(w_{\theta})$. The element $w_{0}$
and $w_{\theta}$ are characterized by the property \cite[Chapitre IV, Exercise 22]{MR0240238}
\begin{eqnarray}
l(ww_{0}) & = & l(w_{0})-l(w),\;\forall w\in W\label{eq:sum}\\
l(ww_{\theta}) & = & l(w_{\theta})-l(w),\;\forall w\in W_{\Theta}
\end{eqnarray}
that imply immediately $w_{0}^{2}=1$ and $w_{\theta}^{2}=1$. It
follows that, for every $w\in W$ 
\[
l(w_{0}w)=l((w_{0}w)^{-1})=l(w^{-1}w_{0}^{-1})=l(w^{-1}w_{0})=l(w_{0})-l(w^{-1})=l(w_{0})-l(w).
\]
Furthermore, set $W^{\Theta}=\left\{ w\in W/l(ws)=l(w)+1\;\forall s\in\Theta\right\} $.
We have, for every $(w,\bar{w})\in W^{\Theta}\times W_{\Theta}$,
\begin{equation}
l(w\bar{w})=l(w)+l(\bar{w}).\label{eq:additivity}
\end{equation}

\begin{prop}
\label{thm:Varieties acted by G are poly}Let $X$ be a smooth $n$-dimensional
variety and let $G$ be an affine group which acts transitively on
$X$. Suppose that, for every $k=1,...,n-1$, there exists a finite
family of subvarieties $\left\{ \Omega_{a}\right\} _{a\in I_{k}}$
of dimension $k$ such that
\begin{enumerate}
\item $\left\langle \left\{ \left[\Omega_{a}\right]/a\in I_{k}\right\} \right\rangle =H_{2k}(X,\Z)$
or $A_{k}(X)$, and
\item $\forall a\in I_{k},\exists b\in I_{n-k}$\textup{ such that $\Omega_{c}\cdot\Omega_{b}=\delta_{a,c}$
$\forall c\in I_{k}$.}
\end{enumerate}
Then $\Neff kX=\Effc kX=\Eff kX$ is polyhedral and simplicial.\end{prop}
\begin{proof}
We will suppose that the classes of the subvarieties $\left\{ \Omega_{a}\right\} _{a\in I_{k}}$
generate $H_{2k}(X,\Z)$, the case $A_{k}(X)$ being similar. Let
$\omega_{a}$ be the class of $\Omega_{a}$ in $N_{k}(X)$. Let $\gamma\in\Neff kX$.
By (\ref{eq:pi_k_R}) there is a class $\beta\in\nicefrac{Z_{k}(X)}{\Hom_{k}(X)}\otimes\R\subseteq H_{2k}(X,\R)$
such that $\pi_{k,\R}(\beta)=\gamma$. By (1) we have that $\beta=\sum_{a\in I_{k}}\gamma_{a}[\Omega_{a}]$
and then $\gamma=\sum\gamma_{a}\pi_{k}([\Omega_{a}])=\sum\gamma_{a}\omega_{a}$.
Let $a\in I_{k}$ and let $b\in I_{n-k}$ be as in (2). Then $\gamma\cdot\omega_{b}=\gamma_{a}\ge0$
because $\gamma$ is nef and $\omega_{b}$ is effective. Therefore
$\gamma\in\Eff kX$, then $\Neff kX\subseteq\Eff kX$. Furthermore,
from $\omega_{c}\cdot\omega_{a}=\delta_{a,c}$ it follows that the
system $\left\{ \omega_{a}\right\} _{a\in I_{k}}$ is linearly independent. Let $A$ a
subvariety of $X$ of dimension $k$, and let $B$ be a subvariety
of $X$ of codimension $k$. By Kleiman's Theorem \cite{MR0360616}
there is an element $g\in G$ such that $gA$ is rationally equivalent
to $A$ and generically transverse to $B$. Then $A\cdot B=(gA)\cdot B=\#((gA)\cap B)\ge0$,
so $\Eff kX\subseteq\Neff kX$. It is clear that $\Neff kX$ is generated
by $\left\{ \omega_{a}/a\in I_{k}\right\} $. Since $\Neff kX$ is
closed and, as seen above, generated by the $\omega_{a}$, we get
that $\Neff kX=\Effc kX$ is polyhedral.\end{proof}
\begin{prop}
\label{prop:Rat.hom.var has polyh}Let X be a rational homogeneous
variety. Then $\Neff kX=\Effc kX=\Eff kX$ is polyhedral.\end{prop}
\begin{proof}
The description of the Chow ring of any rational homogeneous variety
given in \cite[Corollary(1.5)]{MR1092142} is 
\[
A_{*}(X)=\underset{w\in W^{\Theta}}{\bigoplus}\Z[X_{w}]
\]
where $X_{w}$ is the closure of the set $BwP/P$, with dimension
$l(w)$ \cite[Proposition(1.3)]{MR1092142}. Let $I_{k}=\{w\in W^{\Theta}/l(w)=k\}$.
Given $w\in W^{\Theta}$ we claim that $w_{0}ww_{\theta}\in I_{\dim X-k}$.
Indeed for all $s\in\Theta$, using (\ref{eq:sum}) and (\ref{eq:additivity}),
we have 
\[
\begin{array}{c}
l(w_{0}ww_{\theta}s)=l(w_{0})-l(ww_{\theta}s)=l(w_{0})-l(w)-l(w_{\theta}s)\\
=l(w_{0})-l(w)-l(w_{\theta})+l(s)=l(w_{0})-l(ww_{\theta})+1=l(w_{0}ww_{\theta})+1
\end{array}
\]
Similarly we can prove that $l(w_{0}ww_{\theta})=l(w_{0})-l(w_{\theta})-l(w)$.
Now given $w\in I_{k}$ we have, by \cite[Proposition(1.4)]{MR1092142},
that (2) of Proposition \ref{thm:Varieties acted by G are poly} is
satisfied.
\end{proof}
The pseudoeffective cone is also polyhedral in the case when the action
of $G$ on $X$ has finitely many orbits, see \cite[Corollary p.2]{MR1299008}.

Among the rational homogeneous varieties, the following are particularly
interesting.
\begin{defn}
Let $r,s$ be two integers such that $2\le r\le\frac{s}{2}$. The
Grassmann variety of $r$-planes $G(r,s)$ is the scheme of $r$-dimensional
subspaces of $\C^{s}$. Let $\omega$ be a non-degenerate symmetric
bilinear form on $\C^{s}$. The orthogonal Grassmannian of isotropic
$r$-planes $OG(r,s)$ is the scheme of $r$-dimensional subspaces
of $\C^{s}$ isotropic with respect to $\omega$. The scheme $OG(r,2m)$
has two isomorphic connected components if $r=m$ or $m-1$. In these
two cases, we will denote by $OG_{+}(r,2m)$ a connected component
of $OG(r,2m)$. Let $\sigma$ be a non-degenerate symplectic bilinear
form on $\C^{s}$. The symplectic Grassmannian of isotropic $r$-planes
$SG(r,s)$ is the scheme of $r$-dimensional subspaces of $\C^{s}$
isotropic with respect to $\sigma$.\end{defn}
\begin{rem}
\label{rmk:OGSGaszerolocusinG}Let $S$ be the universal subbundle
of $G(r,s)$. The Pl\"{u}cker embedding is the embedding given by the
very ample line bundle $\wedge^{r}S^{\dual}$. The varieties $OG(r,s)$
and $SG(r,s)$ can be embedded in $G(r,s)$ as zero sections of, respectively,
$Sym^{2}S^{\dual}$ and $\wedge^{2}S^{\dual}$.
\end{rem}

\subsubsection{Complete intersection of rational homogeneous varieties}
\begin{rem}
\label{Rmk (2) of OG(k,2k)}In \cite[Proposition 34]{MR3114934},
it is stated that the smooth complete intersection of $OG_{+}(k,2k)$
of type $(2,2)$ under the Pl\"{u}cker embedding is a weak 2-Fano variety.
This should be read as $(2)$.
\end{rem}

\begin{rem}
We introduce the following
notation: the group $H^{4}(G(r,s),\Z)$ is generated by $\{\s_{2},\s_{1,1}\}$,
while $H^{r(s-r)-4}(G(r,s),\Z)$ is generated by a basis $\{\s_{2}^{\dual},\s_{1,1}^{\dual}\}$
dual to $\{\s_{2},\s_{1,1}\}$.
\end{rem}

\begin{rem}
Let $X$ be a smooth complete intersection of $G(2,5)$ of type $(1,1)$
under the Pl\"{u}cker embedding, let $Z$ be the variety of lines through
a general point of $X$. \cite[Example 30]{MR3114934} says that $Z$
has homology class equal to $\sigma_{2}^{\v}+\sigma_{1,1}^{\v}$.
This should be read as $2\sigma_{1,1}^{\v}+\sigma_{2}^{\v}$.
\end{rem}

\begin{rem}
\label{Rmk c(O_G)}By Serre duality $\c(\Omega_{G(2,5)}^{p}(-m))=\c(\Omega_{G(2,5)}^{6-p}(m))$,
and for $m=1,2,3$ we have $\c(\Omega_{G(2,5)}(-m))=\c(\Omega_{G(2,5)}^{5}(m))=0$
because all the groups $H^{p}(G(2,5),\Omega_{G(2,5)}^{5}(m))$ are
zero by \cite[Theorem p. 171(3)]{Snow86}. If $m=1,2$ we have $\c(\Omega_{G(2,5)}^{2}(-m))=\c(\Omega_{G(2,5)}^{4}(m))=0$,
because $\forall p\ge0$ $H^{p}(G(2,5),\Omega_{G(2,5)}^{4}(m))=0$
by \cite[Theorem p.p. 165,169]{Snow86}. It can easily be seen that
$\c(\Omega_{G(2,5)})=-1$ and $\c(\Omega_{G(2,5)}^{2})=2$.
\end{rem}

\begin{lem}
\label{lem:b4(X11)=00003D2}Let $X$ be a smooth complete intersection
of type $(1,1)$ in a Grassmann variety $G(2,5)$ under the Pl\"{u}cker
embedding. Then $b_{4}(X)=2$.\end{lem}
\begin{proof}
By \cite[Example 7.1.5]{MR2095472}, all rows of the Hodge Diamond
of $X$, except the middle row, are equal to those of the Hodge Diamond
of $G=G(2,5)$. Since $X$ is Fano, $\h 04(X)=0$ then
\begin{eqnarray}
\c(\Omega_{X}) & = & -1-\h 13(X)\label{eq:-1-h13}\\
\c(\Omega_{X}^{2}) & = & \h 22(X)\label{eq:h22}\\
b_{4}(X) & = & \h 22(X)+2\h 13(X)\label{eq:betti}
\end{eqnarray}
Note that by Serre duality and adjunction formula, for any integer
$m$ 
\[
h^{4}(\O_{X}(-m))=h^{0}(\O_{X}(m)\otimes\O_{G}(2-5)_{|X})=h^{0}(\O_{X}(m-3))
\]
then by Kodaira Vanishing Theorem, $\c(\O_{X}(-1))=\c(\O_{X}(-2))=0$.
Take the Koszul resolution of the sheaf $\O_{X}$

\begin{equation}
0\ra\O_{G}(-2)\ra\O_{G}(-1)^{\oplus2}\ra\O_{G}\ra\O_{X}\ra0\label{eq:Koszul}
\end{equation}
and tensor it by $\Omega_{G}$
\begin{equation}
0\ra\Omega_{G}(-2)\ra\Omega_{G}(-1)^{\oplus2}\ra\Omega_{G}\ra\Omega_{G|X}\ra0\label{eq:successione}
\end{equation}
then, by Remark \ref{Rmk c(O_G)}, 
\[
\c(\Omega_{G|X})=\c(\Omega_{G}(-2))-2\c(\Omega_{G}(-1))+\c(\Omega_{G})=-1
\]
If we tensor (\ref{eq:successione}) by $\O_{G}(-1)$ we have 
\[
\c(\Omega_{G|X}(-1))=\c(\Omega_{G}(-3))-2\c(\Omega_{G}(-2))+\c(\Omega_{G}(-1))=0
\]
From the canonical sequence

\begin{equation}
0\ra\O_{X}(-1)^{\oplus2}\ra\Omega_{G|X}\ra\Omega_{X}\ra0\label{eq:Canonical}
\end{equation}
we get $\c(\Omega_{X})=\c(\Omega_{G|X})-2\c(\O_{X}(-1))=-1$, then
$\h 13(X)=0$ by (\ref{eq:-1-h13}). If, instead, we tensor (\ref{eq:Koszul})
by $\Omega_{G}^{2}$, that is

\[
0\ra\Omega_{G}^{2}(-2)\ra\Omega_{G}^{2}(-1)^{\oplus2}\ra\Omega_{G}^{2}\ra\Omega_{G|X}^{2}\ra0
\]
we get, by Remark \ref{Rmk c(O_G)}, 
\[
\c(\Omega_{G|X}^{2})=\c(\Omega_{G}^{2}(-2))-2\c(\Omega_{G}^{2}(-1))+\c(\Omega_{G}^{2})=2
\]
By \cite[Exercise II.5.16d]{MR0463157} and (\ref{eq:Canonical})
we get 
\[
\c(\Omega_{X}^{2})=\c(\Omega_{G|X}^{2})-2\c(\Omega_{G|X}(-1))-3\c(\O_{X}(-2))=2
\]
Then by (\ref{eq:h22}) and (\ref{eq:betti}) we get $\h 22(X)=2$
and $b_{4}(X)=2$.\end{proof}
\begin{prop}
Let $X$ be a $n$-dimensional weak 2-Fano complete intersection in
a Grassmann variety $G(r,s)$ under the Pl\"{u}cker embedding. Then, $b_{4}(X)\le2$.
In particular $\Effc 2X$ is polyhedral.\end{prop}
\begin{proof}
Assume that $X$ is of type $(d_{1},...,d_{c})$. If $n>4$, by \cite[Theorem 7.1.1]{MR2095472},
we have $b_{4}(X)=b_{4}(G(r,s))\le2$ and we can apply Lemma \ref{lem:Lemma 1}.
If $n=4$, using \cite[Proposition 31]{MR3114934}, we have the following
conditions: $c=r(s-r)-4$ and $\sum_{i=1}^{c}d_{i}\le s-1$. It is
easy to see that this leads to the following cases

\begin{table}[H]
\centering{}%
\begin{tabular}{|c|c|c|c|c|}
\hline 
$G(r,s)$ & Type &  & $G(r,s)$ & Type\tabularnewline
\hline 
\hline 
$G(2,7)$ & $(1,1,1,1,1,1)$ & \multirow{4}{*}{} & \multirow{4}{*}{$G(2,5)$} & $(1,1)$\tabularnewline
\cline{1-2} \cline{5-5} 
$G(3,6)$ & $(1,1,1,1,1)$ &  &  & $(1,2)$\tabularnewline
\cline{1-2} \cline{5-5} 
\multirow{2}{*}{$G(2,6)$} & $(1,1,1,1)$ &  &  & $(1,3)$\tabularnewline
\cline{2-2} \cline{5-5} 
 & $(1,1,1,2)$ &  &  & $(2,2)$\tabularnewline
\hline 
\end{tabular}
\end{table}

None of them is weak 2-Fano by \cite[Proposition 31 and 32(iv)]{MR3114934},
and Theorem \ref{thm:Gen De Arruda}.
\end{proof}
Now we can prove Theorem \ref{thm:Gen De Arruda}.
\begin{proof}
Let $\O_{Y}(1)$ be the Pl\"ucker line bundle and let 
\[
\U\subseteq\bP(H^{0}(Y,\O_{Y}(1)))\times\bP(H^{0}(Y,\O_{Y}(1)))
\]
be the open set parametrizing the smooth complete intersections in
$Y$ of bidegree $(1,1)$. For $t\in\U$, we denote by $X_{t}$ the
corresponding variety. Let $\mathcal{X}:=\{(x,t)\in Y\times\U:x\in X_{t}\}$
and consider the family
\[
\xymatrix{\mathcal{X}\ar[d]^{pr_{2}}\ar[r]^{pr_{1}} & Y\\
\mathcal{U}
}
\]
Suppose $Y=G(2,5)$. Let $i:X_{t}\ra Y$ be the inclusion, the map
$i^{*}:H^{4}(Y,\Z)\ra H^{4}(X_{t},\Z)$ is injective with torsion
free cokernel by \cite[Theorem 7.1.1 and Example 7.1.2]{MR2095472},
since $b_{4}(Y)=b_{4}(X_{t})=2$ by Lemma \ref{lem:b4(X11)=00003D2},
we have that $i^{*}:H^{4}(Y,\Z)\ra H^{4}(X_{t},\Z)$ is an isomorphism.
By \cite[Corollary 5.1]{DeArruda}, for a general $t$ there
exists a surface $S_{t}$ such that $[i(S_{t})]=\sigma_{1,1}^{\v}$.
Then there exist $a_{t},b_{t}\in\Z$ such that $S_{t}=a_{t}\sigma_{2|X_{t}}+b_{t}\sigma_{1,1|X_{t}}$.
Since
\begin{eqnarray*}
(\sigma_{2|X_{t}})^{2} & = & (\sigma_{2}^{2})\cdot\sigma_{1}^{2}=(\sigma_{3,1}+\sigma_{2,2})\cdot\sigma_{1}^{2}=2\\
(\sigma_{1,1|X_{t}})^{2} & = & (\sigma_{1,1}^{2})\cdot\sigma_{1}^{2}=\sigma_{2,2}\cdot\sigma_{1}^{2}=1\\
\sigma_{2|X_{t}}\cdot\sigma_{1,1|X_{t}} & = & (\sigma_{2}\cdot\sigma_{1,1})\cdot\sigma_{1}^{2}=\s_{3,1}\cdot\sigma_{1}^{2}=1
\end{eqnarray*}
Using the condition $[i(S_{t})]=\sigma_{1,1}^{\v}=\sigma_{2,2}$,
we have
\[
\begin{array}{c}
0=\sigma_{2,2}\cdot\sigma_{2}=S_{t}\cdot\sigma_{2|X_{t}}=2a_{t}+b_{t}\\
1=\sigma_{2,2}\cdot\sigma_{1,1}=S_{t}\cdot\sigma_{1,1|X_{t}}=a_{t}+b_{t}
\end{array}
\]
then $a_{t}=-1$ and $b_{t}=2$. Let $\mathcal{S}:=pr_{1}^{*}(-\sigma_{2}+2\sigma_{1,1})$,
then the surface $\mathcal{S}_{|X_{t}}$ is such that $[S_{t}]=[\mathcal{S}_{|X_{t}}]$,
and since we see that it is effective for a general $t$, hence it
is effective for all%
\footnote{This is a well-known fact for experts. A good reference is \cite[Proposition 3]{Ottem:inpress}.%
} $t$. Let $t\in\U$, then $X_{t}$ is not weak 2-Fano since using
\cite[Proposition 32]{MR3114934}

\[
ch_{2}(X_{t})\cdot\mathcal{S}_{|X_{t}}=\frac{1}{2}(\sigma_{2|X_{t}}-\sigma_{1,1|X_{t}})\cdot(-\sigma_{2|X_{t}}+2\sigma_{1,1|X_{t}})=-\frac{1}{2}.
\]
Suppose $Y=G(2,6)$. By \cite[Theorem 7.1.1]{MR2095472} we have that
$H^{4}(Y,\Z)\cong H^{4}(X_{t},\Z)$, then $b_{8}(X_{t})=b_{4}(X_{t})=2$.
Now consider $i^{*}:H^{8}(Y,\Z)\ra H^{8}(X_{t},\Z)$, where $i:X_{t}\ra Y$
is the inclusion. From

\begin{eqnarray*}
\s_{4|X_{t}}\cdot\s_{2|X_{t}} & = & (\s_{4}\cdot\s_{2})\cdot\s_{1}^{2}=\s_{4,2}\cdot\s_{1}^{2}=\s_{4,3}\cdot\s_{1}=1\\
\s_{2,2|X_{t}}\cdot\s_{2|X_{t}} & = & (\s_{2,2}\cdot\s_{2})\cdot\s_{1}^{2}=\s_{4,2}\cdot\s_{1}^{2}=\s_{4,3}\cdot\s_{1}=1\\
\s_{4|X_{t}}\cdot\s_{1,1|X_{t}} & = & (\s_{4}\cdot\s_{1,1})\cdot\s_{1}^{2}=0\cdot\s_{1}^{2}=0\\
\s_{2,2|X_{t}}\cdot\s_{1,1|X_{t}} & = & (\s_{2,2}\cdot\s_{1,1})\cdot\s_{1}^{2}=(\s_{2,2}\cdot(\s_{1}^{2}-\s_{2}))\cdot\s_{1}^{2}\\
 & = & (\s_{2,2}\cdot\s_{1}^{2}-\s_{2,2}\cdot\s_{2})\cdot\s_{1}^{2}=(\s_{3,2}\cdot\s_{1}-\s_{4,2})\cdot\s_{1}^{2}\\
 & = & (\s_{4,2}+\s_{3,3}-\s_{4,2})\cdot\s_{1}^{2}=(\s_{3,3})\cdot\s_{1}^{2}=\s_{4,3}\cdot\s_{1}=1
\end{eqnarray*}
Hence $\sigma_{4|X_{t}}$ and $\sigma_{2,2|X_{t}}$ are a basis of
$H^{8}(X_{t},\Z)$, since that group is torsion free (see Remark \ref{rmk:h5istorsionfree}).
Then $[S_{t}]=a_{t}\sigma_{4|X_{t}}+b_{t}\sigma_{2,2|X_{t}}$, where
as before $S_{t}$ is the surface described in \cite[Corollary 5.1]{DeArruda}
for general $t\in\mathcal{U}$. Using the condition $[i(S_{t})]=\sigma_{1,1}^{\v}=\sigma_{3,3}$,
we have
\[
\begin{array}{c}
0=\sigma_{3,3}\cdot\sigma_{2}=S_{t}\cdot\sigma_{2|X_{t}}=a_{t}+b_{t}\\
1=\sigma_{3,3}\cdot\sigma_{1,1}=S_{t}\cdot\sigma_{1,1|X_{t}}=b_{t}
\end{array}
\]
then $a_{t}=-1$ and $b_{t}=1$. Let $\mathcal{S}:=pr_{1}^{*}(-\sigma_{4}+\sigma_{2,2})$,
then $[S_{t}]=[\mathcal{S}_{|X_{t}}]$, that is $\mathcal{S}_{|X_{t}}$
is effective for all $t$. Let $t\in\U$, then $X_{t}$ is not weak
2-Fano since using \cite[Proposition 32]{MR3114934}

\[
ch_{2}(X_{t})\cdot\mathcal{S}_{|X_{t}}=(\sigma_{2|X_{t}}-\sigma_{1,1|X_{t}})\cdot(-\sigma_{4|X_{t}}+\sigma_{2,2|X_{t}})=-1.
\]

\end{proof}
\begin{rem}
\label{rmk:h5istorsionfree}By \cite[Theorem 7.1.1]{MR2095472} we have $H^{5}(X_{t},\Z)=0$. By \cite[Corollary 3.3]{MR1867354}
$H_{4}(X_{t},\Z)$ is torsion free, then also $H^{8}(X_{t},\Z)$ is
torsion free by Poincar\'e duality.
\end{rem}
We now deal with complete intersections in orthogonal Grassmannians,
so let us recall the useful notation in \cite{Cos09}. Given a connected
component $X\subseteq OG(r,s)$, we will write $s=2m+1-\epsilon$
with $\epsilon\in\left\{ 0,1\right\} $ and $2\le r\le m$. Let $t$
be an integer such that $0\le t\le r$, and $t\equiv m$ $(mod\,2)$
if $2r=s$. Given a sequence of integers $\lambda=(\lambda_{1},...,\lambda_{t})$
of length $t$ such that 
\[
m-\e\ge\lambda_{1}>...>\lambda_{t}>-\e.
\]
Let $\tilde{\lambda}=(\tilde{\lambda}_{t+1},...,\tilde{\lambda}_{m})$
be the unique sequence of length $m-t$ such that
\begin{itemize}
\item $m-1\ge\tilde{\lambda}_{t+1}>...>\tilde{\lambda}_{m}\ge0$,
\item $\tilde{\lambda}_{j}+\lambda_{i}\ne m-\e$ for every $i=1,..,t$ and
$j=t+1,...,m$.
\end{itemize}
The Schubert varieties in $X$ are parametrized by pairs $(\lambda,\mu)$,
where $\mu$ is any subsequence of $\tilde{\lambda}$ of length $r-t$.
Given an isotropic flag of subvector spaces $F_{\bullet}$

\[
0\subseteq F_{1}\subseteq F_{2}\subseteq...\subseteq F_{m}\subseteq F_{m-1}^{\perp}\subseteq F_{m-2}^{\perp}\subseteq...\subseteq F_{1}^{\perp}\subseteq\C^{s},
\]
$\Omega_{(\lambda,\mu)}(F_{\bullet})$ is defined as the closure of
the locus

\[
\begin{array}{c}
\left\{ \left[W\right]\in X/\dim(W\cap F_{m+1-\e-\lambda_{i}})=i\,\mathrm{for\,1\le i\le t};\right.\\
\left.\dim(W\cap F_{\mu_{j}}^{\perp})=j\,\mathrm{for\,}t<j\le r\right\} .
\end{array}
\]
Let us define another sequence $\lambda'$ of length $t'$ in this
way:
\begin{itemize}
\item $\lambda'=\lambda$ if either $\epsilon=0$ or $\epsilon=1$ and $t\equiv m$
$(mod\,2)$; otherwise
\item $\lambda'=\lambda\cup\{b\}$ where $b=\min\{a\in\N/0\le a\le m-1,\, a\notin\lambda,\, a+\mu_{j}\ne m-1\,\forall j=t+1,...,k\}$.
\end{itemize}
Let $\tilde{\lambda'}$ be the unique sequence associated to $\lambda'$
as above. Then the pair $(\lambda,\mu)$ is a subsequence of $(\lambda',\tilde{\lambda'})$.
Suppose $(\lambda,\mu)=({\lambda}'_{i_{1}},...,{\lambda'}_{i_{t}},\tilde{\lambda'}_{i_{t+1}},...,\tilde{\lambda'}_{i_{r}})$
and let the discrepancy of $\lambda$ and $\mu$ be the non-negative
number 
\[
dis(\lambda,\mu)=\sum_{j=1}^{r}(m-r+j-i_{j}).
\]
Then the codimension of a Schubert cycle $\Omega_{(\lambda,\mu)}(F_{\bullet})$
is (see \cite[p.2448]{Cos09}) 
\[
\mathrm{codim}(\Omega_{(\lambda,\mu)}(F_{\bullet}))=\sum_{i=1}^{t'}\lambda'_{i}+dis(\lambda,\mu).
\]
Let $\Omega_{(\lambda,\mu)}(F_{\bullet})$ be of codimension $k$
and set $\sigma_{(\lambda,\mu)}=\left[\Omega_{(\lambda,\mu)}(F_{\bullet})\right]\in H^{2k}(X,\Z)$.
The set of all $\sigma_{(\lambda,\mu)}$ of codimension $k$ is a
basis of $H^{2k}(X,\Z)$ (by the Ehresmann's Theorem \cite{Ehr34}).
\begin{lem}
\label{lem:Betti OG}Let $X$ be a connected component of $OG(r,s)$,
$2\le r\le m=\left[\frac{s}{2}\right]$, we have
\[
b_{4}(X)=\begin{cases}
1 & r=m\\
3 & 1\le m-r\le2,\, s\, even\\
2 & otherwise
\end{cases}
\]
\end{lem}
\begin{proof}
We have to count the number of sequences $(\lambda,\mu)$ such that
\[
\sum_{i=1}^{t'}\lambda'_{i}+dis(\lambda,\mu)=2.
\]
For $1\le j\le r$ let $c_{j}=m-r+j-i_{j}$. It can easily be seen
that 
\[
m-r\ge c_{1}\ge c_{2}\ge....\ge c_{r}\ge0
\]
and we can write 
\[
dis(\lambda,\mu)=\sum_{i=1}^{r}c_{j}.
\]
We are in one of the following cases:
\begin{enumerate}
\item $\sum_{i=1}^{t'}\lambda'_{i}=0$ and $dis(\lambda,\mu)=2$, or
\item $\sum_{i=1}^{t'}\lambda'_{i}=1$ and $dis(\lambda,\mu)=1$, or
\item $\sum_{i=1}^{t'}\lambda'_{i}=2$ and $dis(\lambda,\mu)=0$.
\end{enumerate}
Let $s$ be odd. Then

\begin{itemize}

\item [Case (1)] $t$ must be 0. If $m-r\ge1$ then $c_{1}=c_{2}=1$,
and, if $m-r>1$,\\
we have also the possibility $c_{1}=2$. These cases correspond to
\[
(\lambda,\mu)=\begin{cases}
(\emptyset,(r,r-1,r-3,...,))\\
(\emptyset,(r+1,r-2,r-3,...,)).
\end{cases}
\]

\item [Case (2)] Only one possibility if $m-r=1$, that is $\lambda=(1)$
and $c_{2}=1$.\\
This case corresponds to $(\lambda,\mu)=((1),(r-2,r-3,...,))$. No
other possibilities if $m-r\ne1$.

\item [Case (3)] It must be $\lambda=(2)$, then $i_{1}=1$ and since
$c_{j}=0$ $\forall j\ge1$, $c_{1}=m-r+1-1=0$ implies $m=r$. This
is the case $(\lambda,\mu)=((2),(m-1,m-3,...))$.\end{itemize}

Let $s$ be even. If $s=2r$, then the discrepancy is 0 because $c_{j}\le m-r$
$\forall j\ge1$, then it is possible only the case 3, that is 
\[
(\lambda,\mu)=\begin{cases}
((2),(m-1,m-2,m-4)) & m\,\mathrm{odd}\\
((2,0),(m-2,m-4)) & m\,\mathrm{even.}
\end{cases}
\]
Suppose $m>r$. Let $m$ be even, then

\begin{itemize}

\item[Case (1)] It must be $\lambda'=\emptyset$, then $\lambda=\lambda'=\emptyset$
and $\tilde{\lambda}=(m-1,m-2,m-3,m-4,...)$. If $m-r\ge1$ then $c_{1}=c_{2}=1$,
and, if $m-r\ge2$, we have also the possibility $c_{1}=2$. These
cases corresponds to 
\[
(\lambda,\mu)=\begin{cases}
(\emptyset,(r,r-1,r-3,...,))\\
(\emptyset,(r+1,r-2,r-3,...,)).
\end{cases}
\]

\item[Case (2)] It must be $\lambda'=(1,0)$, then we can have $\lambda=(0)$
or $\lambda=(1)$.

Suppose $\lambda=(0)$, $\tilde{\lambda}=(m-2,m-3,...)$, and we have
to choose a $\mu$ such that $b=1$ in order to have $\lambda'=\lambda\cup\{1\}$
which implies $\tilde{\lambda}'=(m-3,m-4,...)$. This can happen only
if $m-2\notin\mu$, that is, it is enough to choose $\mu$ as a subsequence
of $(m-3,m-4,...)$. This case implies that $i_{1}=2$, then $c_{1}=m-r+1-2=m-r-1$,
then it must be $m-r=2$. Since $c_{j}=0$ $\forall j\ge2$, that
corresponds to the case 
\[
(\lambda,\mu)=((0),(m-4,m-5,...,)).
\]

Suppose $\lambda=(1)$, $\tilde{\lambda}=(m-1,m-3,...)$, and we have
to choose a $\mu$ such that $b=0$ in order to have $\lambda'=\lambda\cup\{0\}$
which implies $\tilde{\lambda}'=(m-3,m-4,...)$. This can happen only
if $m-1\notin\mu$, that is, it is enough to choose $\mu$ as a subsequence
of $(m-3,m-4,...)$. This case implies that $i_{1}=1$, then $c_{1}=m-r+1-1=m-r$,
then it must be $m-r=1$. Since $c_{j}=0$ $\forall j\ge2$, that
corresponds to the case 
\[
(\lambda,\mu)=((1),(m-3,m-4,m-5,...,)).
\]

\item[Case (3)] It must be $\lambda'=(2,0)$, then we can have $\lambda=(0)$
or $\lambda=(2)$.

If $\lambda=(2)$, then $c_{1}=m-r$, then the discrepancy is not
0.

So $\lambda=(0)$, $\tilde{\lambda}=(m-2,m-3,...)$, $c_{j}=0$ $\forall j\ge1$,
and we have to choose a $\mu$ such that $b=2$ in order to have $\lambda'=\lambda\cup\{2\}$
which implies $\tilde{\lambda}'=(m-2,m-4,...)$. This can happen only
if $m-2\in\mu$ and $m-3\notin\mu$. That is, the sequence 
\[
((0),\mu)=((0),(\tilde{\lambda'}_{i_{1}},...,\tilde{\lambda'}_{i_{r}}))
\]
seen as a subsequence of $((2,0),(m-2,m-4,...))=(\lambda',\tilde{\lambda'})$
must satisfy $i_{1}=2$. The condition $c_{j}=0$ implies $i_{j}=m-r+j$,
then $i_{1}=m-r+1=2$ implies $m-r=1$. Then, if $m-r=1$, we have
the sequence 
\[
(\lambda,\mu)=((0),(m-2,m-4,...)).
\]

\end{itemize}

Let $m$ be odd, then

\begin{itemize}

\item[Case (1)] It must be $\lambda'=(0)$, then we can have $\lambda=\lambda'=(0)$
or $\lambda=\emptyset$.

Suppose $\lambda=\lambda'=(0)$, this implies $\tilde{\lambda}=(m-2,m-3,m-4,...)$
and $c_{1}=m-r$. Then

-if $m-r\ge3$, then this case in not possible since the first summand
of the discrepancy (which it must be 2) is $m-r$,

-if $m-r=2$, then $c_{j}=0$ for $j\ge2$, that is $i_{j}=m-r+j$
for $j\ge2$, then 
\[
(\lambda,\mu)=((0),(\tilde{\lambda}_{m-r+2},\tilde{\lambda}_{m-r+3},...,))=((0),(r-2,r-3,...)),
\]

-if $m-r=1$, then $c_{j}=0$ for $j\ge3$ and $c_{2}=1$, that is
\[
(\lambda,\mu)=((0),(\tilde{\lambda}_{m-r+1},\tilde{\lambda}_{m-r+3},...,))=((0),(r-1,r-3,...)).
\]

Suppose $\lambda=\emptyset$, $\tilde{\lambda}=(m-1,m-2,...)$, and
we have to choose a $\mu$ such that $b=0$ in order to have $\lambda'=\lambda\cup\{0\}$
which implies $\tilde{\lambda}'=(m-2,m-3,m-4,...)$. This can happen
only if $m-1\notin\mu$, that is, it is enough to choose $\mu$ as
a subsequence of $(m-2,m-3,m-4,...)$. If $m-r\ge1$ we have $c_{1}=c_{2}=1$,
that corresponds to the case 
\[
(\lambda,\mu)=(\emptyset,(r,r-1,r-3,...,)).
\]
But, in order to make $m-1\notin\mu$, we must have $r\neq m-1$,
then this case only happen if $m-r\ge2$. If $m-r\ge2$, we have also
the possibility $c_{1}=2$, that corresponds to the case 
\[
(\lambda,\mu)=(\emptyset,(r+1,r-2,r-3,...,)).
\]
But, in order to make $m-1\notin\mu$, $r+1\neq m-1$, then this case
only happen if $m-r\ge3$.

\item[Case (2)] It must be $\lambda'=(1)$, then we can have $\lambda=\lambda'=(1)$
or $\lambda=\emptyset$. 

Suppose $\lambda=\lambda'=(1)$, then $\tilde{\lambda}=(m-1,m-3,m-4,...)$,
$c_{1}=m-r$, and $c_{j}=0$ for $j\ge2$. So, if $m-r=1$, we have
the sequence 
\[
(\lambda,\mu)=((1),(\tilde{\lambda}_{m-r+2},\tilde{\lambda}_{m-r+3},...,))=((1),(m-3,m-4,...)).
\]

Suppose $\lambda=\emptyset$, $\tilde{\lambda}=(m-1,m-2,...)$, $c_{1}=1$,
$c_{j}=0$ $\forall j\ge2$, and we have to choose a $\mu$ such that
$b=1$ in order to have $\lambda'=\lambda\cup\{1\}$ which implies
\[
\tilde{\lambda}'=(m-1,m-3,m-4,...).
\]
This can happen only if $m-1\in\mu$ and $m-2\notin\mu$. That is,
the sequence 
\[
(\emptyset,\mu)=(\emptyset,(\tilde{\lambda'}_{i_{1}},...,\tilde{\lambda'}_{i_{r}}))
\]
seen as a subsequence of 
\[
((1),(m-1,m-3,m-4,...))=(\lambda',\tilde{\lambda'})
\]
must satisfy $i_{1}=2$. The condition $c_{1}=1$ implies $1=m-r+1-i_{1}$,
then $1=m-r+1-2$ that is $m-r=2$, while the condition $c_{j}=0\,\forall j\ge2$
implies $i_{j}=m-r+j$. Then, if $m-r=2$, we have the sequence 
\[
(\lambda,\mu)=((\emptyset),(m-1,m-4,m-5,...)).
\]

\item[Case (3)] It must be $\lambda'=(2)$, then we can have $\lambda=\lambda'=(1)$
or $\lambda=\emptyset$.

If $\lambda=(2)$, then $c_{1}=m-r$, then the discrepancy is not
0.

So $\lambda=\emptyset$, $\tilde{\lambda}=(m-1,m-2,...)$, $c_{j}=0$
$\forall j\ge1$, and we have to choose a $\mu$ such that $b=2$
in order to have $\lambda'=\lambda\cup\{2\}$ which implies 
\[
\tilde{\lambda}'=(m-1,m-2,m-4,...).
\]
This can happen only if $m-1,m-2\in\mu$ and $m-3\notin\mu$. That
is, the sequence 
\[
(\emptyset,\mu)=(\emptyset,(\tilde{\lambda'}_{i_{1}},...,\tilde{\lambda'}_{i_{r}}))
\]
seen as a subsequence of 
\[
((2),(m-1,m-2,m-4,...))=(\lambda',\tilde{\lambda'})
\]
must satisfy $i_{1}=2$ and $i_{2}=3$. The condition $c_{j}=0$ implies
$i_{j}=m-r+j$, then $i_{1}=m-r+1=2$ and $i_{2}=m-r+2=3$ imply $m-r=1$.
Then, if $m-r=1$, we have the sequence 
\[
(\lambda,\mu)=((\emptyset),(m-1,m-2,m-4,...)).
\]

\end{itemize}\end{proof}
\begin{lem}
\label{lem:b6OG(r,2r)}$b_{6}(OG_{+}(r,2r))=2$.\end{lem}
\begin{proof}
We have to calculate the number of Schubert cycles of dimension 6,
that is the number of sequences $r-1\ge\lambda_{1}>...>\lambda_{t}\ge0$
such that $\sum_{i=1}^{t}\lambda_{i}=3$, $t\equiv r$ $(mod\,2)$.
We get
\begin{itemize}
\item If $r$ is odd, $\lambda=(3)$ and $\lambda=(2,1,0)$;
\item If $r$ is even, $\lambda=(3,0)$ and $\lambda=(2,1)$.
\end{itemize}
\end{proof}
We now deal with complete intersections in symplectic Grassmannians
$SG(r,s)$ with $2\le r\le m=\frac{s}{2}$. Let us recall the useful notation
in \cite{Cos13}. Let $t$ be an integer such
that $0\le t\le r$. Given a sequence of integers $\lambda=(\lambda_{1},...,\lambda_{t})$
of length $t$ such that 
\[
m\ge\lambda_{1}>...>\lambda_{t}>0
\]
let $\tilde{\lambda}=(\tilde{\lambda}_{t+1},...,\tilde{\lambda}_{m})$
be the unique sequence of length $m-t$ such that
\begin{itemize}
\item $m-1\ge\tilde{\lambda}_{t+1}>...>\tilde{\lambda}_{m}\ge0$,
\item $\tilde{\lambda}_{j}+\lambda_{i}\ne m$ for every $i=1,..,t$ and
$j=t+1,...,m$.
\end{itemize}
The Schubert varieties in $SG(r,s)$ are parametrized by pairs $(\lambda,\mu)$,
where $\mu$ is any subsequence of $\tilde{\lambda}$ of length $r-t$.
Given an isotropic flag of subvector spaces $F_{\bullet}$

\[
0\subseteq F_{1}\subseteq F_{2}\subseteq...\subseteq F_{m}\subseteq F_{m-1}^{\perp}\subseteq F_{m-2}^{\perp}\subseteq...\subseteq F_{1}^{\perp}\subseteq\C^{s}
\]
$\Omega_{(\lambda,\mu)}(F_{\bullet})$ is defined as the closure of
the locus

\[
\begin{array}{c}
\{\left[W\right]\in SG(r,s)/\dim(W\cap F_{m+1-\lambda_{i}})=i\,\mathrm{for\,1\le i\le t};\\
\dim(W\cap F_{\mu_{j}}^{\perp})=j\,\mathrm{for\,}t<j\le r\}.
\end{array}
\]
Suppose $(\lambda,\mu)=({\lambda}{}_{1},...,{\lambda}_{t},\tilde{\lambda}_{i_{t+1}},...,\tilde{\lambda}_{i_{r}})$,
the codimension of $\Omega_{(\lambda,\mu)}(F_{\bullet})$ is (see
\cite[p. 57]{Cos13}) 
\[
\mathrm{codim}(\Omega_{(\lambda,\mu)}(F_{\bullet}))=\sum_{i=1}^{t}\lambda{}_{i}+dis(\lambda,\mu).
\]
The set all $\sigma_{(\lambda,\mu)}=\left[\Omega_{(\lambda,\mu)}(F_{\bullet})\right]$
of codimension $k$ is a basis of $H^{2k}(SG(r,s),\Z)$ by Ehresmann's
Theorem. The proof of the following lemma is the same of the case
of $OG(r,2m+1)$.
\begin{lem}
\label{lem:Betti SG}Let $2\le r\le m=\frac{s}{2}$, then 

\[
b_{4}(SG(r,s))=\begin{cases}
2 & m-r\ge1\\
1 & r=m
\end{cases}
\]

\end{lem}

\subsection{Other examples}
\begin{prop}
\label{prop:c.i. in OG}Let $s,r$ be positive integers such that
$2\le r\le\left[\frac{s}{2}\right]$, and $\left[\frac{s}{2}\right]-r\ne1,2$
if $s$ is even. Let $s\ne2r$ (respectively, $s=2r$), let $X$ be
a $n$-dimensional weak 2-Fano complete intersection in a connected
component of the orthogonal Grassmann variety $OG(r,s)$ under the
Pl\"{u}cker (respectively, half-spinor) embedding, with X very general
if $X\subseteq OG(2,7)$. Then $\Effc 2X$ is polyhedral.\end{prop}
\begin{proof}
Assume that $X$ is of type $(d_{1},...,d_{c})$. If $n>4$, by \cite[Theorem 7.1.1]{MR2095472}
and Lemma \ref{lem:Betti OG}, we have $b_{4}(X)\le2$ and we can
apply Lemma \ref{lem:Lemma 1}. Then we have $n=4$ and $c=\frac{r(2s-3r\text{\textminus}1)}{2}-4$.
If $2r=s$, by \cite[Proposition 34]{MR3114934} and Remark \ref{Rmk (2) of OG(k,2k)},
we see that $X$ is weak 2-Fano if and only if either $d_{i}=1$ and
$c\le4$, or $X$ of type $(2)$. Therefore we get $r=4$ and $X$
of type $(1,1)$. By \cite[Proposition 34]{MR3114934} we have that
$K_{X}=-c_{1}(X)=-4H$, where $H$ is the half-spinor embedding. But
then, by \cite[Corollary p.37]{MR0316745}, $X$ is a smooth quadric
in $\bP^{5}$ and then $b_{4}(X)=2$ by \cite[p.20]{Reid:thesis},
so we apply by Lemma \ref{lem:Lemma 1}.

If $2r\neq s$, since $c_{1}(OG(r,s))=(s-r-1)\sigma_{1}$ we get that
$\sum_{i=1}^{c}d_{i}\le s-r-2$. It is easy to see that this leads
to the following cases

\begin{table}[H]
\centering{}%
\begin{tabular}{|c|c|}
\hline 
$OG(r,s)$ & Type\tabularnewline
\hline 
\hline 
$OG(3,7)$ & $(1,1)$\tabularnewline
\hline 
$OG(2,7)$ & $(1,1,1)$\tabularnewline
\hline 
\multirow{2}{*}{$OG_{+}(2,6)$} & $(2)$\tabularnewline
\cline{2-2} 
 & $(1)$\tabularnewline
\hline 
\end{tabular}
\end{table}

But $OG(3,7)\cong OG_{+}(4,8)$, then the first case is a quadric.
Let $X_{111}$ be the variety $(1,1,1)$ in $OG(2,7)$. This is the
variety (b8) in the classification given in \cite{MR1326986}. Indeed,
for the reader's convenience, we point out that $X_{111}$ is the
zero-locus of a global section of the bundle
\[
\left(\wedge^{2}S^{\dual}\right)^{\oplus3}\oplus Sym^{2}S^{\dual}
\]
where $S^{\dual}$ is $(1,0;0,0,0,0,0)$ in K\"{u}chle's notation (see
\cite[Section 2.5]{MR1326986}). So $h^{1,3}(X_{111})>0$ by \cite[Theorem 4.8]{MR1326986}.
Now apply \cite[Theorem 2]{MR1378596} to conclude that the space
of algebraic cycles of $X_{{111}}$ is induced by the space of algebraic
cycles of $OG(2,7)$. Then 
\[
\nicefrac{Z_{2}(X_{111})}{\mathrm{Alg}_{2}(X_{111})}\otimes\R
\]
is at most 2-dimensional. Hence $\Effc 2{X_{111}}$ is polyhedral
by (\ref{eq:Diagramma}) and Lemma \ref{lem:Lemma 1}. The last two
varieties do not satisfy the condition $\left[\frac{s}{2}\right]-r\ne1,2$,
anyway, they are not weak 2-Fano by \cite[Example 21]{MR3114934}.
Indeed, $OG_{+}(2,6)$ is the zero section of the bundle $\O_{\P 3}(1)\oplus\O_{\P 3}(1)$
in $\P 3\times\P 3$ \cite[Proposition 2.1]{MR3397419}, and it can
easily be seen that the Pl\"{u}cker embedding is given by the divisor
$(1,1)$, then the two varieties are isomorphic to, respectively,
a complete intersection of type $(1,1)$ and $(1,2)$ in $\P 3\times\P 3$
under the embedding given by $\O_{\P 3}(1)\oplus\O_{\P 3}(1)$.\end{proof}
\begin{prop}
Let $X$ be a smooth $n$-dimensional weak 2-Fano complete intersection
in a symplectic Grassmann variety $SG(r,s)$ under the Pl\"{u}cker embedding.
Then, $b_{4}(X)\le2$. In particular $\Effc 2X$ is polyhedral.\end{prop}
\begin{proof}
Assume that $X$ is of type $(d_{1},...,d_{c})$. If $n>4$, by \cite[Theorem 7.1.1]{MR2095472}
and Lemma \ref{lem:Betti SG}, we have $b_{4}(X)=b_{4}(SG(r,s))\le2$
and we can apply Lemma \ref{lem:Lemma 1}. If $n=4$, since $c_{1}(SG(r,s))=(s-r+1)\sigma_{1}$
we have the following conditions: $c=\frac{r(2s-3r+1)}{2}-4$ and
$\sum_{i=1}^{c}d_{i}\le s-r$. It is easy to see that this leads to
the following cases:

\begin{table}[H]
\centering{}%
\begin{tabular}{|c|c|}
\hline 
$SG(r,s)$ & Type\tabularnewline
\hline 
\hline 
\multirow{2}{*}{$SG(3,6)$} & $(1,1)$\tabularnewline
\cline{2-2} 
 & $(1,2)$\tabularnewline
\hline 
\multirow{2}{*}{$SG(2,6)$} & $(1,1,1)$\tabularnewline
\cline{2-2} 
 & $(1,1,2)$\tabularnewline
\hline 
\end{tabular}
\end{table}

The variety $SG(2,6)$ is a section of $\wedge^{2}(S^{\dual})=\O_{G(2,6)}(1)$,
as we said in Remark \ref{rmk:OGSGaszerolocusinG}. Thus the last
two case are, respectively, $(1,1,1,1)$ and $(1,1,1,2)$ in $G(2,6)$.
The first two cases are not weak 2-Fano by \cite[Proposition 36]{MR3114934},
the last two by \cite[Proposition 32(i)]{MR3114934}.
\end{proof}

\section{Fano manifolds of dimension $n$ and index $i_{X}>n-3$}

A very important invariant of a Fano variety $X$ is its index: this
is the maximal integer $i_{X}$ such that $-K_{X}$ is divisible by
$i_{X}$ in $Pic(X)$.

Fano varieties of high index have been classified: \cite{MR0316745}
proved that $i_{X}\le n+1$, $i_{X}=n+1$ if and only if $X=\P n$,
and $i_{X}=n$ if and only if $X\subset\P{n+1}$ is a smooth hyperquadric.
Furthermore the case $i_{X}=n-1$ (the so called Del Pezzo varieties)
has been classified by Fujita in \cite{MR664549,MR676113}, and the
case $i_{X}=n-2$ (the so called Mukai varieties) by Mukai (see \cite{MR995400}
and \cite{MR1668579}).

Araujo and Castravet \cite[Theorem 3]{MR3114934} succeeded to classify
2-Fano Del Pezzo and Mukai varieties. They proved:
\begin{thm}
\label{thm:ClassAC}Let X be a 2-Fano variety of dimension $n\ge3$
and index $i_{X}\ge n-2$. Then $X$ is isomorphic to one of the following.\\
\begin{minipage}[t]{1\columnwidth}%
\begin{itemize}
\item $\Pn$.
\item Complete intersection in projective spaces:\\
- Quadric hypersurfaces $X\subseteq\P{n+1}$ with $n>2$;\\
- Complete intersections of type $(2,2)$ in $\P{n+2}$ with $n>5$;
\\
- Cubic hypersurfaces $X\subseteq\P{n+1}$ with $n>7$; \\
- Quartic hypersurfaces $X\subseteq\P{n+1}$ with $n>15$; \\
- Complete intersections of type $(2,3)$ in $\P{n+2}$ with $n>11$;
\\
- Complete intersections of type $(2,2,2)$ in $\P{n+3}$ with $n>9$.
\item Complete intersection in weighted projective spaces:\\
- Degree 4 hypersurfaces in $\bP(2,1,...,1)$ with $n>11$; \\
- Degree 6 hypersurfaces in $\bP(3,2,1,...,1)$ with $n>23$; \\
- Degree 6 hypersurfaces in $\bP(3,1,...,1)$ with $n>26$; \\
- Complete intersections of type $(2,2)$ in $\bP(2,1,...,1)$ with
$n>14$.
\item $G(2,5)$.
\item $OG_{+}(5,10)$ and its linear sections of codimension $c<4$.
\item $SG(3,6)$.
\item $G_{2}/P_{2}$.\end{itemize}
\end{minipage}
\end{thm}
Here $G_{2}/P_{2}$ is a 5-dimensional homogeneous variety for a group
of type $G_{2}$. Using the results in the previous sections we can prove Theorem \ref{Teorema1}.

\begin{proof}
For $\Effc 2X$: In the case $\Pn$ and its complete intersections,
we can invoke Theorem \ref{thm:Conj for comp int in Pn}. Since none
of the complete intersections in $\Pw$ of the list has dimension
4, we can use Proposition \ref{prop:3}. Also $G(2,5)$, $OG_{+}(5,10)$,
$SG(3,6)$ and $G_{2}/P_{2}$ are rational homogeneous varieties,
then their cone of pseudoeffective 2-cycles is polyhedral by Proposition
\ref{prop:Rat.hom.var has polyh}. Whereas the complete intersections
of $OG_{+}(5,10)$ have polyhedral cone of pseudoeffective 2-cycles
by Proposition \ref{prop:c.i. in OG}.

For $\Effc 3X$: In Theorem \ref{thm:ClassAC}, the only complete
intersections of dimension 6 in a weighted projective space are the
one of type $(2,2)$ in $\bP^{8}$ and the smooth quadric $Q\subseteq\P 7$.
The first one is not weak 3-Fano since by \cite[Equation (3.1)]{MR3114934},
$ch_{3}(X)=-\frac{7}{6}h_{|X}^{3}$ where $h$ is the class of an
hyperplane in $\P 8$. Then $h_{|X}^{3}$ is effective, and $ch_{3}(X)\cdot h_{|X}^{3}=-\frac{7}{6}h_{|X}^{6}<0$.
For the quadric, by \cite[p.20]{Reid:thesis} $b_{6}(Q)=2$, then
$\Effc 3X$ is polyhedral by Lemma \ref{lem:Lemma 1}.%
{} For the other complete intersections we can use Proposition \ref{prop:3},
whilst for the rational homogeneous varieties we can use Proposition
\ref{prop:Rat.hom.var has polyh}. Also for the complete intersections
in $OG_{+}(5,10)$ we have $b_{6}(X)=2$, because $b_{6}(OG_{+}(5,10))=2$
by Lemma \ref{lem:b6OG(r,2r)} and we can use \cite[Theorem 7.1.1]{MR2095472}.
\end{proof}

\bibliographystyle{amsalpha}
\providecommand{\bysame}{\leavevmode\hbox to3em{\hrulefill}\thinspace}
\providecommand{\MR}{\relax\ifhmode\unskip\space\fi MR }
\providecommand{\MRhref}[2]{%
  \href{http://www.ams.org/mathscinet-getitem?mr=#1}{#2}
}
\providecommand{\href}[2]{#2}

\end{document}